\documentclass[10pt]{amsart}
\usepackage{amsmath, amsthm, amsfonts, amssymb}
\usepackage{mathrsfs}
\providecommand{\noopsort[1]{}}
\usepackage{bbm}
\usepackage{dsfont}
\usepackage{ifthen}
\numberwithin{equation}{section}
\usepackage{a4}
\usepackage[active]{srcltx}
\usepackage{verbatim}
\usepackage{hyperref}

 \usepackage[fleqn,tbtags]{mathtools}
 \mathtoolsset{showonlyrefs,showmanualtags}

\newtheorem{thm}{Theorem}[section]

\newtheorem{lem}[thm]{Lemma}

\theoremstyle{remark}
\newtheorem{rem}[thm]{Remark}

\newtheorem{example}[thm]{Example}

\theoremstyle{definition}
\newtheorem{defn}[thm]{Definition}

\renewcommand{\div}{\mathrm{div}\,}

\newcommand{\eps}{\varepsilon}

\newcommand{\one}{\mathbbm{1}}

\newcommand{\form}[3]{\ifthenelse{\equal{#2}{}}{\mbox{$ #1\Big[\, \cdot\,  , \, \cdot\,  \Big]$}}{
\mbox{$ #1\Big[ #2 , #3 \Big]$}}}
\newcommand{\qform}[2]{\ifthenelse{\equal{#2}{}}{\mbox{$ #1\Big[\cdot \Big]$}}{
\mbox{$ #1\Big[ #2 \Big]$}}}
\newcommand{\ip}[2]{\ifthenelse{\equal{#1}{}}{\mbox{$ \Big( \,\cdot\; \vline \; \cdot \, \Big) $}}{
\mbox{$ \Big( #1 \;  \vline \; #2 \Big)$}}}
\newcommand{\norm}[1]{\ifthenelse{\equal{#1}{}}{\mbox{$\|\cdot\|$}}{\mbox{$\| #1 \|$}}}
\newcommand{\betr}[1]{\ifthenelse{\equal{#1}{}}{\mbox{$|\cdot|$}}{\mbox{$| #1 |$}}}
\newcommand{\dual}[2]{\ifthenelse{\equal{#1}{}}{\mbox{$ \langle \cdot,\cdot \, \rangle $}}{
\mbox{$ \langle #1 , #2 \rangle$}}}
\newcommand{\pdual}[2]{\ifthenelse{\equal{#1}{}}{\mbox{$ \Big[ \,\cdot\; , \; \cdot \, \Big] $}}{
\mbox{$ \Big[ #1 \;  , \; #2 \Big]$}}}
\newcommand{\bdual}[2]{\ifthenelse{\equal{#1}{}}{\mbox{$ \Big\langle \,\cdot\; , \; \cdot \, \Big\rangle_* $}}{
\mbox{$ \Big\langle #1 \;  , \; #2 \Big\rangle_*$}}}

\newcommand{\CR}{\mathds{R}}

\newcommand{\CN}{\mathds{N}}

\newcommand{\sgn}{\mathrm{sgn}\,}

\newcommand{\PP}{\mathbf{P}}
\renewcommand{\P}{\mathbb{P}}

\newcommand{\A}{\mathcal{A}}

\newcommand{\bS}{{S}}

\newcommand{\cF}{\mathscr{F}}

\newcommand{\expect}{\mathbb{E}}
\newcommand{\cL}{\mathscr{L}}

\newcommand{\half}{\frac{1}{2}}

\newcommand{\inject}{\hookrightarrow}

\renewcommand{\P}{{\mathds{P}}}

\newcommand{\FF}{\mathds{F}}

\newcommand{\OO}{\mathcal{O}}

\newcommand{\g}{\gamma}

\let\mathcal \undefined
\def\mathcal{\mathscr}

\allowdisplaybreaks

\begin{document}

\title[Stochastic Reaction-Diffusion Systems]
{Stochastic Reaction-Diffusion Systems with H\"older Continuous Multiplicative Noise}

\author{Markus C. Kunze}

\begin{abstract}
We prove pathwise uniqueness and strong existence of solutions for stochastic reaction-diffusion systems with 
locally Lipschitz continuous reaction term of polynomial growth and H\"older continuous multiplicative noise. Under additional 
assumptions on the coefficients, we also prove positivity of the solutions.
\end{abstract}

\address{Graduiertenkolleg 1100, University of Ulm, 89069 Ulm, Germany}
\email{markus.kunze@uni-ulm.de}

\thanks{The author was supported by the Deutsche Forschungsgesellschaft in the framework of the DFG research training group 1100.}
\maketitle

\section{Introduction}

Reaction-diffusion systems and stochastic perturbations of them play an important role in applications in chemistry, biology and physics \cite{murray}. 
In an abstract form, a stochastic reaction-diffusion system can be treated as a stochastic evolution equation
\begin{equation}\label{eq.scp}
dU(t)  =  \big[ AU(t) + F(U(t))\big] dt + G(U(t)) dW_H(t)
\end{equation}
on a Banach space $E$, which is a space of $\CR^r$-valued functions, defined on a domain $\OO$. 
Thus, \eqref{eq.scp} actually represents a system of $r$ coupled equations.
Here,  $A$ is a diagonal matrix of second order differential operators which describe the diffusion in the system and 
 the map $F$ accounts for the reaction in the system and is typically a composition operator with components of polynomial growth.  
 The system is driven by a cylindrical Wiener process $W_H$ in a suitable Hilbert space $H$.\smallskip

In the case of $r=1$, i.e.\ a single reaction-diffusion equation rather than a system, there  are many articles concerned with 
such equations, both in the case of additive noise (see \cite{dpz92, gp93, gp93a, dprzab2}) and in the case of (locally) Lipschitz 
continuous multiplicative noise (see \cite{bgp94, bp99, mz99, p95, KvN12}). Stochastic reaction-diffusion \emph{systems} 
with locally Lipschitz continuous multiplicative noise were considered in \cite{Cerrai}.\smallskip 

In the case where the noise term $G$ is no longer locally Lipschitz continuous the techniques from the above references can no longer be used. 
This is essentially due to the fact that for such equations a priori only stochastically weak solutions (or martingale solutions) can be constructed.
However, stochastic reaction-diffusion equations with merely H\"older continuous multiplicative noise appear naturally, e.g.\ in scaling limits 
of interacting particle systems, see e.g.\ \cite{mt95} were for $r=1$ noise terms with $G(u)(x) := |u(x)|^\half$ and $G(u)(x) = |u(x)(1-u(x))|^\half \one_{\{ 0\leq u(x) \leq 1\}}$ appear.

For such stochastic reaction-diffusion systems with H\"older continuous multiplicative noise only few results are available and, to the best of 
our knowledge, only the case $r=1$ is treated.
In \cite{bg99}, existence of solutions has been proved under an additional boundedness assumption on the coefficient $G$. However, a uniqueness 
result is missing in that article, except for the case of locally Lipschitz continuous $G$.  Pathwise uniqueness for the stochastic heat equation on $\CR^d$
(i.e.\ $r=1$, $A$ is the Laplace operator and $F\equiv 0$)  was proved in \cite{mps06} in the case where $W_H$ is replaced with a colored 
noise and $G$ is composition with a $\gamma$-H\"older continuous function, where the allowed values of $\gamma$ depend on the noise.
In \cite{mp11}, the one-dimensional version of the stochastic heat equation with white noise was considered, and it was proved that for that 
equation pathwise uniqueness holds for $\gamma > 3/4$. For $\gamma < 3/4$, it was recently proved in \cite{mmp12} that solutions are neither pathwise unique nor unique in law.\medskip

In this article, we are concerned with stochastic reaction-diffusion systems on a bounded Lipschitz domain $\OO \subset \CR^d$ of the form
\begin{equation}\label{eq.rds}
\left\{ \begin{array}{lll}
d u_l (t, x) & = & \big[ \A_l u_l(t,x) + f_l (x, u_1(t,x) , \ldots, u_r(t,x))\big] dt \\[0.2em]
&& \quad + \sum_{k=1}^\infty g_{l,k}(x, u_l(t,x)) d\beta_{l,k}(t), \\
&& \quad t > 0, x \in \OO, l=1, \ldots, r\\[0.2em]
\frac{\partial u_l}{\partial \nu_{\A_l}}(t,x) & = & 0 \quad t > 0, x \in \partial \OO, l=1, \ldots, r\\
u_l(0,x) & = & \xi_l (x) \quad x \in \overline{\OO}, l=1, \ldots, r\, . 
\end{array}
\right. 
\end{equation}
Here, $\A_l$ is a uniformly elliptic, second order differential operator in divergence form on $\OO$ and $\partial/\partial\nu_{A_l}$ is the associated conormal derivative. The functions $f_l : \OO \times \CR^k \to \CR$ are locally Lipschitz continuous, of polynomial growth and satisfy suitable dissipativity assumptions. Typical examples that we have in mind are odd-degree polynomials with negative leading coefficients, see 
Section \ref{s.fhn}. The functions
$g_{l,k} : \OO \times \CR \to \CR$ are locally $\half$-H\"older continuous and of linear growth such that the H\"older constants and the 
coefficients in the linear growth are square summable. Finally, the $\beta_{l,k}$ are independent, one-dimensional Brownian motions.
By \cite{mmp12}, pathwise uniqueness cannot be expected for equations driven by space-time white noise with $\half$-H\"older continuous coefficient. However,  our assumption allow us to consider noise terms of the form $G(U(t))RdW$, where $G$ is composition with a locally $\half$-H\"older continuous function of linear growth , $R$ is a Hilbert-Schmidt operator on $L^2(\OO)$ satisfying additional assumptions and 
$W_H$ is space-time white noise, see Example \ref{ex.r}.

We will make our assumptions precise in Section \ref{sect.prelim}\medskip

Under these assumptions, we prove pathwise uniqueness (Theorem \ref{t.uniqueness}) of solutions to equation \eqref{eq.rds} on the state space $E= C(\overline{\OO})^r$. Our proof follows the ideas of Yamada and Watanabe \cite{yw1}, who proved pathwise uniqueness for finite-dimensional 
SDE with $\half$-H\"older continuous coefficients.  The main difficulty to extend these results to the infinite-dimensional setting is of course 
to handle the differential operators $\A_l$. The strategy from \cite{mp11, mps06} to prove pathwise uniqueness, which is also an adaption of the Yamada-Watanabe ideas, cannot be used in our situation. Indeed, there the authors convolute solutions $u$ of the stochastic heat equation with a mollifier $\varphi_n$. In their variational framework, this yields the term 
$u\ast \Delta \varphi_n$ in the equation for the resulting process. It is then used that, as a consequence of its translation invariance, the Laplacian
commutes with convolutions, i..e.\ we have $u\ast \Delta \varphi_n = \Delta (u\ast \varphi_n)$. This is no longer true for differential operators 
with nonconstant coefficients, which we consider here. We overcome this difficulty by using the concept of an \emph{(analytically) weak solution}, see 
Definition \ref{def.weak}, which allows us to perform pointwise estimates.

We would like to point out that in the Yamada-Watanabe result, it is essential that the multiplicative noise is diagonal.
Therefore, in \eqref{eq.rds} we have also considered ``diagonal noise'' by letting the noise term
in the $l$-th equation only depend only on $u_l$, rather than the whole vector $u$. 
Thus, the equations in \eqref{eq.rds} are only coupled via the reaction terms $f_l$

Under additional assumptions on the functions $f_l$ and $g_{l,k}$, similar arguments as in the proof of Theorem \ref{t.uniqueness}
can be used to show that the solutions to \eqref{eq.rds} preserve positivity, i.e.\ if $\xi_1, \ldots, \xi_n \geq 0$ a.s.\ then we also have $u_l(t) \geq 0$ almost surely for all $t \geq 0$ and $l=1, \ldots, r$, see Theorem \ref{t.positive}.

Using pathwise uniqueness, we can adopt the strategy from \cite{Cerrai, KvN12} to prove existence of solutions to \eqref{eq.rds} 
in Theorem \ref{t.ex2}. In contrast to the existence result from \cite{bg99}, we can drop the uniform boundedness assumption on 
$G$ and allow $G$ of linear growth.\smallskip 

This article is organized as follows. In Section \ref{sect.prelim}, we fix our assumptions on equation \eqref{eq.rds} and rewrite it in the abstract 
form \eqref{eq.scp}. We also prove some preliminary results and recall some stochastic concepts that we will use. Section \ref{s.pathwise} contains the prove of pathwise uniqueness and Section \ref{s.positive} our result about positivity of solutions.
Existence of solutions will be proved in Section \ref{s.existence}. In the concluding Section \ref{s.fhn}, we apply our results 
to a stochastic reaction-diffusion system of Fitzhugh-Nagumo type.

\section{Preliminaries}\label{sect.prelim}
In this section, we fix our assumptions on equation \eqref{eq.rds}, rewrite the equation in the abstract form 
\eqref{eq.scp} and recall different notions of existence and uniqueness for equation \eqref{eq.scp} that will be used 
in what follows.

Throughout, all vector spaces are real. If $H$ is a Hilbert space, we write $(\cdot, \cdot)_H$ for the inner product in $H$.  When
$P_n(u)$ and $Q_n(u)$ are certain quantities depending on an index $n$ and a function $u$, we write $P_n(u) \lesssim Q_n(u)$ to indicate
that there exists a constant $c$, independent of $n$ and $u$ such that $P_n(u) \leq c Q_n (u)$ for all $n$ and $u$. We write 
$P_n (u) \eqsim Q_n (u)$ if both $P_n(u) \lesssim Q_n (u)$ and $Q_n (u) \lesssim P_n (u)$. 

\subsection{The differential operators}

We assume that the differential operators $\A_l$ are given by 
\[
\A_l = \sum_{i,j=1}^d D_i \big( a^l_{ij} D_j\big)  - c_l 
\]
where we make the following assumptions:
\begin{itemize}
\item[(A)]
The domain $\OO \subset \CR^d$ is bounded and has Lipschitz boundary. For $l=1, \cdots, r$, the matrix valued functions 
$a_l := (a^l_{ij}) : \overline{\OO} \to \CR^{d\times d}$ are symmetric and have measurable entries. Moreover, for certain $\eta, M > 0$ we have
\[
\eta |\xi|^2 \leq \sum_{i,j =1}^d a_{i,j}(x) \xi_i \xi_j \leq M |\xi|^2
\]
for almost all $x \in \OO$ and all $\xi \in \CR^d$. Finally, $c_l \in L^\infty (\OO)$.
\end{itemize}

To construct realizations of the differential operators which are generators of strongly continuous semigroups, we follow a variational approach.
We consider on $L^2(\OO)$ the symmetric form
\[
\mathfrak{a}^l [u,v] := \int_\OO \sum_{i,j=1}^da_{i,j}^l(x) D_i u(x) \overline{D_j u(x)} \, dx + \int_\OO u(x)v(x) c_l(x) \, dx
\]
endowed with domain $\mathsf{D}(\mathfrak{a}^l) := H^1(\OO)$. The associated operator $A_{l,2}$ is given by 
\[
\mathsf{D}(A_{l,2}) := \{ u \in H^1(\OO)\, : \, \, \exists\,  w \in L^2(\OO)\,\, \mbox{s.t.}\, 
(w,v)_{L^2} = \mathfrak{a}^l[u,v]\,\, \mbox{for all}\, v \in H^2(\OO) \}
\]
and $A_{l,2}u = -w$. Note that the boundary condition $\partial u /\partial \nu_{\A_l} := a_l\nabla u \cdot \nu = 0$ on $\partial \OO$ is incorporated in the domain of $A_{l,2}$. As is well-known, the operator $A_{l,2}$ generates a strongly continuous and analytic semigroup 
 $\bS_{l,2}$, see \cite{ouhabaz}. Changing the function $c_l$ by a constant if necessary, we may and shall assume that $\bS_{l,2}$ is uniformly exponentially 
stable, hence $A_{l,2}$ is invertible. Note that we can compensate the change in $c_l$ by appropriately changing $f_l$ in equation \eqref{eq.rds}, 
thus this assumption means no loss of generality for the stochastic reaction-diffusion equation \eqref{eq.rds}.

Since the form $\mathfrak{a}$ is sub-Markovian, see \cite[Chapter 4]{ouhabaz}, the semigroup $\bS_{l,2}$ restricts to a strongly continuous and analytic semigroup $\bS_{l,p}$ on 
$L^p (\OO)$ for every $p \in [2, \infty)$. The generator $A_{l,p}$ of the restricted semigroup $\bS_{l,p}$ is exactly the part of $A_{l,2}$ in $L^p(\OO)$. 

Finally, $\mathsf{D}(A_{l,2}) \subset C(\overline{\OO})$ and the semigroup $\bS_{l,2}$ restricts to a strongly continuous semigroup $\bS_{l,C}$ on 
$C(\overline{\OO})$, whose generator $A_{l,C}$ is the part of $A_{l,2}$ in $C(\overline{\OO})$. In the case where $c_l \equiv 0$, this follows from 
the results of \cite{ft96}, the case of general $c_l \in L^\infty(\OO)$ follows with a perturbation argument.

To prove existence of solutions to \eqref{eq.rds} in spaces of continuous functions, we will need some embedding results for the fractional domain spaces $\mathsf{D}((-A_{l,p})^\alpha)$. Often, such results are obtained from precise knowledge of $\mathsf{D}(A_{l,p})$ and 
interpolation. However, this strategy requires more regularity of the coefficients $a_l$ and the boundary of $\OO$. Here, we follow 
a different approach.

\begin{lem}\label{l.fdinject}
Assume (A) and let $d < 2p$. Then for $\alpha \in (\frac{d}{2p}, 1)$, the fractional domain space $\mathsf{D}((-A_{l,p})^\alpha)$ is continuously 
embedded into $C(\overline{\OO})$.
\end{lem}

\begin{proof}
Let us first note that since $H^1(\OO) \inject L^\frac{2d}{d-2}(\OO)$ for $d \geq 3$ (and $H^1(\OO) \inject L^q(\OO)$ for arbitrary $q \in (1, \infty)$ for $d \leq 2$), it follows from \cite[Section 7.3]{arendtsurv} that  the semigroups $\bS_{l,p}$ are ultracontractive, i.e.\ for $p \geq 2$ the operator $S_{l,p}(t)$ maps $L^p(\OO)$ into $L^\infty (\OO)$ and we have 
\[
\| S_{l,p}(t)\|_{\cL (L^p(\OO), L^\infty (\OO))} \leq c t^{-\frac{d}{2p}} \quad \forall\, 0<t \leq 1
\]
for a certain constant $c= c(l,p) > 0$.

The fractional domain space $\mathsf{D}((-A_{l,p})^\alpha)$ is the range of the (bounded) operator $(-A_{l,p})^{-\alpha}$, defined by 
\begin{equation}\label{eq.afrac}
(-A_{l,p})^{-\alpha}f := \frac{1}{\Gamma (\alpha)} \int_0^\infty t^{\alpha -1} S_{l,p}(t)f \, dt,
\end{equation}
endowed with the norm $\| (-A_{l,p})f\|_{\mathsf{D}((-A_{l,p})^\alpha)} := \|f\|_p$.  We note that since $\bS_{l,p}$ is analytic, $S_{l,p}(t)f \in 
\mathsf{D}(A_{l,p})$ for all $t>0$ and $f \in L^p(\OO)$. As a consequence of \cite{ft96}, $\mathsf{D}(A_{l,p}) \subset C(\overline{\OO})$, hence 
the integrand takes values in $C(\overline{\OO})$.

Thus, to prove that $(-A_{l,p})^{-\alpha}f \in C(\overline{\OO})$, 
it suffices to show that the integral in \eqref{eq.afrac} exists as a Bochner integral in $C(\overline{\OO})$. Hence we have to show that 
$\|t^{\alpha -1} S_{l,p}(t)f\|_\infty$ is integrable on $(0,\infty)$. From ultracontractivity we obtain
\[
\|t^{\alpha -1} S_{l,p}(t)f\|_\infty \leq c t^{\alpha - 1 - \frac{d}{2p}}\|f\|_p \quad \forall\, 0<t\leq 1
\]
which is integrable on $(0,1)$ since $\alpha > \frac{d}{2p}$. Since $\bS_{l,p}$ is uniformly exponentially stable, we obtain 
for a certain constant $c'>0$ that
\begin{eqnarray*}
\|t^{\alpha -1} S_{l,p}(t)f\|_\infty & \leq&  \|S_{l,p}(1)\|_{\cL (L^p(\OO), L^\infty (\OO))} \|S_{l,p}(t-1)\|_{\cL (L^p(\OO), L^\infty (\OO))} \|f\|_p\\
&\leq &c' e^{-\omega t} \|f\|_p \quad \, \forall t> 1
\end{eqnarray*}
which is integrable on $(1, \infty)$. It thus follows that the integral in \eqref{eq.afrac} indeed exists as a Bochner integral in $C(\overline{\OO})$.
Moreover, we  have the estimate
\[
\| (-A)^{-\alpha}f\|_\infty \lesssim \|f\|_p = \| (-A_{l,p})f\|_{\mathsf{D}((-A_{l,p})^\alpha)} \, .
\]
This finishes the proof.
\end{proof}


Let us now introduce the Banach space $E$ and the operator $A$ that will be used in the abstract formulation \eqref{eq.scp} 
of equation \eqref{eq.scp}. We set $E := (C(\overline{\OO}))^r$, endowed with the norm 
$\| u\|_E := \sum_{l=1}^r \|u_j\|_\infty$. The operator $A$ will be the diagonal operator $\mathrm{diag}(A_{1,C}, \ldots, A_{r,C})$, 
with domain $\mathsf{D}(A_{1,C}) \times \cdots \times \mathsf{D}(A_{r,C})$. Then $A$
generates the strongly continuous semigroup $\bS_C := \mathrm{diag} (\bS_{1,C}, \cdots, \bS_{r,C})$. 

To simplify notation, we will drop the index $C$ from now on and merely write $A_1, \ldots, A_r$ and $S_1, \ldots, S_r$ in what follows.

%

\subsection{The reaction term}
Concerning the functions $f_l$, we make the following assumptions:
\begin{itemize}
\item[(F)] For $l=1, \ldots, r$, the function $f_l : \overline{\OO} \times \CR^k \to \CR$ is given as 
\[
f_l (x, s_1, \ldots, s_k ) := h_l (x, s_l) + k_l (x, s_1, \ldots, s_k)
\]
where the continuous function $k_l$ is assumed to be
locally Lipschitz continuous and of linear growth in the variables $s_1, \ldots, s_r$, uniformly with respect to the first, i.e.\ there are constants 
$c_1, c_2$ and $L_m$ for $m \in \CN$
such that
\[
|k_l (x, s_1, \ldots, s_r)| \leq c_1 + c_2\sum_{j=1}^r |s_j| \quad \forall\, x \in \overline{\OO}, s_j \in \CR
\]
and 
\[
|k_l (x, s_1, \ldots, s_r) - k_l (x, t_1, \cdots, t_r) | \leq L_m\sum_{j=1}^r |s_j- t_j| \quad \forall\, x \in \overline{\OO},  s_j, t_j \in [-m,m] \, .
\]
The continuous function
$h_l : \overline{\OO}\times \CR \to \CR$ is assumed to be locally Lipschitz continuous in the second variable, uniformly with respect to the first.
Moreover, we assume that
\begin{enumerate}
\item For a certain constant $a_l>0$ and an integer $N_l$, we have
\[
-a_l (1+|s|^{2N_l+1}\one_{\{s \geq 0\}}) \leq h_l (x, s) \leq a_l (1+|s|^{2N_l+1}\one_{\{s \leq 0\}})
\]
for all $x \in \overline{\OO}$ and $s \in \CR$.
\item For certain constants $a_{1,l}, a_{2,l} \in \CR$ and $b_{1,l},b_{2,l} >0$ and the integer $N_l$ from (F1) we have
\[
a_{1,l} -b_{1,l}s^{2N_l+1} \leq h_l(x,s) \leq a_{2,l} -b_{2,l}s^{2N_l+1}\, .
\]
\end{enumerate}
\end{itemize}

\begin{example}\label{ex.f}
Conditions (F1) and (F2) are satisfied for functions of the form
\[
h_l(x,s) := \sum_{j=1}^{2N_l+1} \omega_{l,j}(x) s^j\, 
\]
where $N_l$ is an integer and the coefficients $\omega_{l,j}$ belong to $C(\overline{\OO})$ and the highest order coefficients $\omega_{l,2N_{l+1}}$
satisfy $\omega_{l,2N_l+1} (x) \leq - \eps < 0$
for some $\eps>0$ and all $x \in \overline{\OO}$. The proof of this fact can be found in Examples 4.2 and 4.5 in \cite{KvN12}.
\end{example}

We now define the operators $H_l : C(\overline{\OO}) \to C(\overline{\OO})$ by $(H_l(u))(x) := h_l (x, u(x))$ and the operators 
$K_l : E \to C(\overline{\OO})$ by $(K_l (u_1, \ldots, u_r))(x) := k_l (x, u_1(x), \ldots, u_r(x))$. Finally, the map $F : E \to E$ which appears
in \eqref{eq.scp} will be given by $F(u) = (F_1(u), \cdots, F_r(u))$, where $F_l (u) = H_l (u_l) + K_l (u)$. Obviously, the maps $F, F_l, H_l, K_l$ 
are locally Lipschitz continuous. 

Conditions (F1) and (F2) imply that the maps $H_l$ satisfy additional dissipativity  conditions which play a crucial role in proving 
existence of solutions of \eqref{eq.scp}. These dissipativity conditions were first used in \cite{dpz92}. 

Recall, that in a Banach space $(B, \|\cdot \|)$, the \emph{subdifferential of the norm at $x$} is given by 
\[
\partial \|x\| := \{ \varphi \in B^* \, : \, \| \varphi \| = 1 \quad \mbox{and} \quad \langle x, \varphi \rangle = \|x\| \} \, .
\]
As a consequence of the Hahn-Banach theorem, $\partial \|x\| \neq \emptyset$ for all $x \in B$. We have

\begin{lem}\label{l.fdiss}
Let $H_l$ be defined as above, $u,v \in C(\overline{\OO})$ and $\varphi \in \partial \|u\|_\infty$. Then we have
\begin{enumerate}
\item $\langle H_l(u+v), \varphi \rangle \leq a'(1+\|v\|_\infty )^{2N_l+1}$ and
\item $\langle H_l(u+v) - H_l(v), \varphi \rangle \leq a''(1+\|v\|_\infty)^{2N_l+1} - b'' \|u\|^{2N_l+1}_\infty$.
\end{enumerate}
for suitable constants $a', a'', b'' >0$ depending only on the constants $a_l$ from (F1) and $a_{1,l}, a_{2,l}, b_{1,l}, b_{2,l}$ from (F2).
\end{lem}

\begin{proof}
See \cite[Section 4.3]{dpz92}, cf.\ also Examples 4.2 and 4.5 in \cite{KvN12}
\end{proof}

\subsection{The noise term} Let us now turn to the stochastic term in equation \eqref{eq.rds}. We shall assume:
\begin{itemize}
\item[(G)] For $l=1, \cdots, r$ and $k \in \CN$ the continuous functions $g_{l,k} : \overline{\OO} \times \CR \to \CR$
have the following properties
\begin{enumerate}
\item There exist $\alpha_l = (\alpha^l_k)_{k\in \CN}, \beta_l = (\beta^l_k)_{k\in \CN} \in \ell^2$ such that 
$|g_{l,k}(x,s)| \leq \alpha^l_k + \beta^l_k|s|$
for all $x \in \overline{\OO}$ and $s \in \CR$.
\item For every $m \in \CN$ there exist continuous functions $\sigma_{l,k, m}: [0, \infty) \to [0, \infty)$ such that
\begin{enumerate}
\item We have $|g_{l,k}(x, s_1) - g_{l,k}(x,s_2)| \leq \sigma_{l,k, m}(|s_1-s_2|)$ for all $x \in \overline{\OO}$ and $s_1, s_2 \in [-m, m]$.
\item For $l=1, \cdots, r$ and $m \in \CN$, the function $\rho_{l, m} : (0, \infty) \to (0,\infty)$, defined as $\rho_{l,m} (s) := \sum_{k=1}^\infty \sigma_{l,k, m}(s)^2$
is increasing in $s$ and satisfies
\[
\int_{0+} \frac{1}{\rho_{l,m}(s)} ds = \infty\, .
\]
\end{enumerate}
\end{enumerate}
\end{itemize}

\begin{example}\label{ex.r}
Let us give an example of functions $g_{l,k}$ which satisfy (G). For simplicity, we assume that $r=1$ and drop the subscript $l$.

Let $R : L^2(\OO) \to L^2(\OO)$ be a Hilbert-Schmidt operator which is diagonalized by an orthonormal basis $(e_k)_{k \in \CN}$, i.e.\
 for some sequence $\lambda = (\lambda_k)_{k \in \CN}
 \in \ell^2$, we have 
\[
Ru = \sum_{k=1}^\infty \lambda_k (u,e_k)_{L^2}e_k \,.
\]
We additionally assume that $e_k \in C(\overline{\OO})$ and that the sequence $\|\lambda_k e_k\|_\infty$ is square summable.
Moreover, let $g: \CR \to \CR$ be of linear growth and locally H\"older continuous of order $\half$, i.e.\ 
$|g(s)| \leq a +b|s|$ for certain $a,b \in \CR$ and there exist constants $c_m>0$ such that
$|g(s_1) - g(s_2)| \leq c_m|s_1-s_2|^\half$ for all $s_1, s_2 \in [-m, m]$.

Then (G) is satisfied if we set $g_k(x,s) := g(s)\lambda_k e_k(x)$. Indeed, (1) holds true with $\alpha_k = a\|\lambda_ke_k\|_\infty$ and 
$\beta_k = b\|\lambda_ke_k\|_\infty$.  For (2), we choose $\sigma_{k,m} (s) = c_m\|\lambda_ke_k\|_\infty s^\half$. Then clearly (a) is satisfied. In (b), we find 
$\rho_{m} (s) = s \sum_{k=1}^\infty \|\lambda_ke_k\|_\infty^2$, hence $\rho_{m}^{-1}$ is not integrable near 0.

In this situation, the noise term in equation \eqref{eq.rds} can actually be rewritten as 
$g(u(t,x)) RdW(t,x)$, where $W$ is a space-time white noise, i.e.\ an $L^2(\OO)$-cylindrical Brownian motion.
\end{example}

\begin{rem}
If the orthonormal basis $(e_k)$ in example \ref{ex.r} is uniformly bounded in $C(\overline{\OO})$, i.e.\ $\sup_k\|e_k\|_\infty < \infty$, 
then our assumptions reduce to $(\lambda_k) \in \ell^2$. This is for example the case if we consider the standard orthonormal basis of 
$L^2(0,2\pi)$.
If we want to consider the orthonormal basis $(e_k)$ which diagonalizes 
the operator $A_2$ on $L^2(\OO)$, then this basis consists of functions in $C(\overline{\OO})$, as $\mathsf{D}(A_s) \subset 
C(\overline{\OO})$, however the functions are typically not uniformly bounded in $C(\overline{\OO})$. Typically, one can obtain 
a bound $\|e_k\|_\infty \lesssim k^s$ for a certain $s> 0$ from ultracontractivity. In this case, one has to color the noise more by 
requiring that $(k^s\lambda_k) \in \ell^2$.
\end{rem}

Let us return to our general setting and explain how the stochastic term is modeled in the abstract equation \eqref{eq.scp}. We define $G_l : C(\overline{\OO}) \to 
\cL (\ell^2, C(\overline{\OO}))$ by 
\[
\big[G_l(u)h \big](x) := \sum_{k=1}^\infty h_k \big[G_{l,k}(u)\big](x)\, ,
\]
where $G_{l,k} : C(\overline{\OO}) \to C(\overline{\OO})$ is defined by $\big[G_{l,k}(u)\big](x) := g_{l,k}(x, u(x))$.
Note that the above series converges in $C(\overline{\OO})$, since $\|h_k g_{l,k}(\cdot, u(\cdot))\|_\infty \leq h_k(\alpha_{l,k} + \beta_{l,k}\|u\|_\infty)$
is summable.

For the purpose of stochastic integration, we will view $G_l$ as a function taking values in $\cL (\ell^2, L^p(\overline{\OO}))$ 
by embedding $C(\overline{\OO})$ into $L^p(\overline{\OO})$. It turns out that $G_l$ even takes values in 
$\gamma (\ell^2, L^p(\OO))$, the space of \emph{$\gamma$-radonifying operators from $\ell^2$ to $L^p(\OO)$}. 

Let us briefly recall the definition of $\gamma$-radonifying operators. Given a Hilbert space $\mathscr{H}$ and a Banach space 
$B$, every finite rank operators $T: \mathscr{H}\to B$ can be represented in the form
\[
T = \sum_{j=1}^N h_j\otimes x_j
\]
for some integer $N\geq 1$, where $(h_j)_{j=1}^N$ is an orthonormal system in $\mathscr{H}$ and $x_1, \ldots, x_N \in B$. Here, 
$h\otimes x$ is the rank one operator mapping $g \in \mathscr{H}$ to $(g,h)_\mathscr{H}x$. With $T$ represented in this form, we define 
\[
\|T\|_{\gamma (\mathscr{H}, B)} := \Big( \expect \Big\| \sum_{j=1}^N \gamma_j x_j\Big\|^2 \Big)^\half
\]
where $(\gamma_{j})_{j=1}^N$ is a sequence of independent standard normal random variables. This norm is independent of the representation 
of $T$ in the above form. The space $\gamma (\mathscr{H}, B)$ is the completion of the finite rank operators 
with respect to the norm $\|\cdot\|_{\gamma (\mathscr{H}, B)}$. The identity operator on $\mathscr{H}\otimes B$ extends to a continuous 
embedding $\gamma (\mathscr{H}, B) \inject \cL (\mathscr{H}, B)$. We may thus view $\gamma (\mathscr{H}, B)$ as a linear subspace of 
$\cL (\mathscr{H}, B)$. The operators belonging to $\gamma (\mathscr{H}, B)$ are called the $\gamma$-radonifying operators 
from $\mathscr{H}$ to $B$. For more information about $\gamma$-radonifying operators, we refer the reader to \cite{vNsurvey}.

To see that the map $G_l$ from above takes values in $\gamma (\ell^2, L^p(\OO))$, let
$(e_k)$ denote the canonical orthonormal basis of $\ell^2$. Then the series 
\[
\sum_{k=1}^\infty | G_l(u)e_k|^2 = \sum_{k=1}^\infty |g_{l,k}(\cdot, u(\cdot))|^2
\]
converges in $L^p(\OO)$. Thus, by \cite[Lemma 2.1]{vNVW08}, $G(u) \in \gamma (\ell^2, L^p(\OO))$ and 
\[
\| G(u)\|_{\gamma (\ell^2, L^p(\OO))} \simeq \Big\| \Big( \sum_{k=1}^\infty |g_{l,k}(\cdot, u(\cdot))|^2\Big)^\half \Big\|_{L^p(\OO)}
\leq |\OO| (\|\alpha_l\|_{\ell^2} + \|\beta_l\|_{\ell^2} \|u\|_\infty)\,,
\]
proving that $G_l$ is of linear growth. Let us next show that $G_l : C(\overline{\OO}) \to \gamma (\ell^2, L^p(\OO))$ is continuous.
To that end, let $u_n \to u$ in $C(\overline{\OO})$. Employing \cite[Lemma 2.1]{vNVW08} a second time, it follows that 
\[
\|G(u_n) - G(u)\|_{\gamma(\ell^2, L^p(\OO))} \simeq 
\Big\| \Big( \sum_{k=1}^\infty |g_{l,k}(\cdot, u_n(\cdot))- g_{l,k}(\cdot, u(\cdot))|^2\Big)^\half \Big\|_{L^p(\OO)} \, ,
\]
which converges to $0$ as $n \to \infty$ by dominated convergence.

Let us summarize the properties of $G_l$

\begin{lem}\label{l.propg}
Assume (G). Then the maps $G_l : C(\overline{\OO}) \to \gamma (\ell^2, L^p(\OO))$ are well-defined, of linear growth and continuous.
\end{lem}

We now proceed to model the stochastic term in equation \eqref{eq.scp}. The driving process $W_H$ is an $H$-cylindrical Wiener process
for a suitable Hilbert space $H$, defined on a filtered probability space $(\Omega, \Sigma, \FF, \P)$, i.e.\ a bounded linear operator 
from $L^2([0,\infty); H)$ to $L^2(\Omega)$ with the following properties:
\begin{enumerate}
\item for  $f \in L^2([0,\infty); H)$, the random variable $W_H(f)$ is centered Gaussian.
\item for  $t>0$ and $f \in L^2([0,\infty); H)$ with support in $[0,t]$, $W_H(f)$ is $\cF_t$-measurable.
\item for $t>0$ and $f \in L^2([0,\infty); H)$ with support in $[t,\infty)$, $W_H(f)$ is independent of $\cF_t$.
\item for $f_1, f_2 \in L^2([0,\infty); H)$ we have $\expect (W_H(f_1)W_H(f_2)) = (f_1, f_2)_{L^2([0,\infty); H)}$.
\end{enumerate}
It is easy to see that for $h\in H$, the process $(W_H(t)h)_{t \geq 0}$, defined by $W_H(t)h:= W_H(\one_{(0,t]}\otimes h)$ is a real valued
$\FF$-Brownian motion (which is standard if $\|h\|_H=1$). Moreover, two such Brownian motions $(W_H(t)h_1)_{t\geq 0}$ and $(W_H(t)h_2)_{t\geq 0}$
are independent if and only $h_1$ and $h_2$ are orthogonal in $H$. We refer to \cite{vNsurvey} for a further discussion.\smallskip

To model the reaction-diffusion system \eqref{eq.rds} in abstract form, we choose $H := \big[\ell^2\big]^r$. Denoting the 
canonical orthonormal basis of $\ell^2$ by $(e_k)$, we put $e_{l,k}:= (0, \cdots, 0, e_k, 0, \ldots, 0)$, where the $e_k$ is at position $l$.
Then $(e_{l,k})$ is an orthonormal basis of $H$. Let $\beta_{l,k}$, $l=1, \ldots, r$ and $k \in \CN$ be a family of independent real-valued 
Brownian motions defined on a common probability space $(\Omega, \Sigma, \P)$. Then $W_H : L^2(\CR_+, H) \to L^2(\Omega)$, defined
by 
\[
W_H (f) := \sum_{l=1}^r\sum_{k=1}^\infty \int_0^\infty (f(t), e_{l,k})_H d\beta_{l,k}(t)
\]
is an $H$-cylindrical Wiener process, see \cite{vNsurvey}. We denote by $P_l : H \to H$ the projection onto the $l$-th component and 
define $G: E \to \cL(H, \tilde{E})$ by 
\[G(u)h := (G_1(u_1)P_1h, \ldots, G_r(u_r)P_rh) .\]
Later on, we will also write $H_l := P_lH$ and define $W_{H_l}$ by $W_{H_l}(f) := W_H(P_lf)$ for $f \in L^2([0,\infty); H)$. Then 
$W_{H_l}$ is an $H_l$-cylindrical Wiener process.

\subsection{Solution concepts}
Having rewritten equation \eqref{eq.rds} in the abstract form \eqref{eq.scp}, we now define what we mean by a solution of 
equation \eqref{eq.scp}, thus by a solution of equation \eqref{eq.rds}. Since the map $G$ is not locally Lipschitz continuous, we initially
work with  stochastically weak solutions, i.e.\ the probability space is part of the solution. We first consider a solution concept which is also 
analytically weak.

\begin{defn}\label{def.weak}
A \emph{weak solution} of equation \eqref{eq.scp} is a tupel $((\Omega, \Sigma, \P), \FF, W_H, u)$, where 
$(\Omega, \Sigma, \P)$ is a probability space endowed with a filtration $\FF$ which satisfies the usual conditions, $W_H$ is 
an $H$-cylindrical Wiener process with respect to $\FF$ and $u = (u(t))_{t \geq}$ is a continuous, $\FF$-progressive, $E$-valued 
process such that for all $x^* \in D(A^*)$ and $t \geq 0$ we have
\begin{equation}\label{eq.weaksol}
\dual{u(t)}{x^*} = \dual{u(0)}{x^*} + \int_0^t\dual{u(s)}{A^*x^*}+ \dual{F(u(s))}{x^*}\, ds + \int_0^t G(u(s))^*x^*\, dW_H(s)\, .
\end{equation}
$\P$-almost surely. 

If an initial datum $\xi \in E$ is specified, we say that $((\Omega, \Sigma, \P), \FF, W_H, U)$ is a weak solution 
to the initial value problem corresponding to \eqref{eq.scp}, if it is a weak solution of \eqref{eq.scp} and additionally $\P(U(0) =\xi ) = 1$.
\end{defn}

Note that the stochastic integral in \eqref{eq.weaksol} is an integral of an $H\simeq H^*$-valued stochastic process with respect to a 
cylindrical Wiener process. Such an integral is defined as follows. If $(h_k)_{k \in \CN}$ is an orthonormal Basis of $H$, put $w_k := W_H(s)h_k$
which is a real valued Brownian motion. We then define
\[
\int_0^t G(u(s))^*x^*\, dW_H(s) := \sum_{k=1}^\infty \int_0^t (G(u(s))^*x^*, h_k)_H \, dw_k(s)\, .
\]
In the special situation of equation \eqref{eq.rds}, 
$\mathsf{D}(A^*)$ can be identified with $\mathsf{D}(A_1^*) \times \cdots \times \mathsf{D}(A_r^*)$ 
which is a subset of $\mathscr{M}(\overline{\OO})^r$, the $r$-fold product of the bounded Borel measures on $\overline{\OO}$. 
Using the standard basis $(e_{l,k})$ of $H = \big[ \ell^2\big]^r$ and writing $\beta_{l,k} := W_He_{kl}$ as in the previous section,
we have for $x^* = (\mu_1, \ldots, \mu_r) \in \mathsf{D}(A^*)$
\[
\int_0^t G(u(s))^*x^*\, dW_H (s) = \sum_{l=1}^r\sum_{k=1}^\infty \int_{\overline{\OO}} g_{l,k}(x, u_l(s,x)) \, d\mu_l (x) \, d\beta_{l,k}(s)
\]
In particular, choosing $\mu_l := R(\lambda, A_l)^*\delta_x$, where $\delta_x$ is Dirac measure in $x$ and $\mu_j = 0$ for $j \neq l$, 
equation \eqref{eq.weaksol} reduces to 
\begin{eqnarray*}
\big[R(\lambda, A_l) u_l (t)\big](x) &= & u_l (0, x) + \int_0^t \big[A_lR(\lambda, A_l)u_l(s)\big](x) + \big[F_l (u(s))\big](x) \, ds\\
&&\quad  + \sum_{l=1}^r  \big[R(\lambda, A_l) G_{l}(u(s))\big](x)\, dW_{H_l}(s).
\end{eqnarray*}
This will be used in the following section.\medskip 

In order to prove \emph{existence} of solutions, the concept of weak solutions is not suitable. Instead, we will use the concept of a \emph{mild}
solution. To define the concept of a mild solution, we have to use a Banach space valued stochastic integral. We will use the theory of 
stochastic integration in UMD Banach space \cite{vNVW07}. We note that our state space $E$ is \emph{not} a UMD Banach space.
However,  the fractional domain space $\mathsf{D}((-A_{l,p})^\alpha)$, being isomorphic to $L^p(\OO)$, is a UMD Banach space. 
We may thus perform stochastic integration in fractional domain spaces and then use Lemma \ref{l.fdinject} to get back into our state space.
More precisely, we have

\begin{lem}\label{l.stochint}
Assume (A) and (G). Moreover, let $(\Omega, \Sigma, \FF, \P)$ be a stochastic basis on which an $H$-cylindrical Wiener process is defined.
Then, if $X= (X_1, \ldots, X_r)$ is a continuous, $\FF$-progressive $E$-valued process, then for every $t> 0$ the process 
$t \mapsto S(t-s)G(X(s))$ is stochastically integrable and the stochastic integral 
\[
\int_0^t S(t-s)G(X(s))\, dW_H(s)
\]
defines an $\cF_t$-measurable, $E$-valued random variable.
\end{lem}

\begin{proof}
It suffices to prove the assertion for $r=1$. We will drop the index $l$ for ease of notation. Fix $\omega \in \Omega$, $t>0$, $p> \max\{2, d/4\}$ and $\alpha \in (\frac{d}{2p}, \half )$. We claim that the map $R_\omega : L^2(0,t; H) \to \mathsf{D}((-A_{p})^\alpha)$ given by
\[
\langle R_\omega f , x^* \rangle := \int_0^t \langle S(t-s)G(X(s, \omega)), x^*\rangle \, ds
\]
for all $x^* \in \mathsf{D}((-A_{p})^\alpha)^*$ is $\gamma$-radonifying. Since $\mathsf{D}((-A_{p})^\alpha)$, being isomorphic to 
$L^p(\Omega)$, has type 2, it suffices to show that $s \mapsto S(t-s)G(X(s, \omega))$ belongs to $L^2(0,t; \gamma (H, \mathsf{D}((-A_{p})^\alpha))$, 
see \cite[Theorem 11.6]{vNsurvey}. To that end, we have
\begin{eqnarray*}
&&\| S(t-\cdot)G(X(\cdot, \omega))\|_{L^2(0,t; \gamma (H, \mathsf{D}((-A_{p})^\alpha))}^2\\
 & = & 
\int_0^t \|S(t-s) G(X(s, \omega))\|_{\gamma (H, \mathsf{D}((-A_{p})^\alpha))}^2\, ds\\
& \leq  & \int_0^t (t-s)^{-2\alpha} \|G(X(s, \cdot))\|_{\gamma (H, \mathsf{D}((-A_{p})^\alpha))}^2\, ds\\
& \leq & \int_0^t (t-s)^{-2\alpha} (|\OO|\|\alpha\|_{\ell^2} + \|\beta\|_{\ell^2}\|X(\cdot, \omega)\|_\infty)^2\, ds < \infty,
\end{eqnarray*}
since $\alpha < \half$.
Here, we have used the estimate $\|S(t)\|_{\cL (L^p(\OO), \mathsf{D}((-A_p)^\alpha))} \lesssim t^{-\alpha}$ and the ideal property for $\gamma$-radonifying operators \cite[Theorem 6.2]{vNsurvey} in the second line and the linear growth of $G$ from Lemma \ref{l.propg} in the third.

It now follows from \cite[Lemma 2.7 and Remark 2.8]{vNVW07}, that $s \mapsto S(t-s)G(X(s,\omega))$ represents a strongly 
measurable map $R: \Omega \to \gamma (L^2(0,t; H); \mathsf{D}((-A_{p})^\alpha))$. As $s \mapsto S(t-s)G(X(s,\omega))$ is $\FF$-progressive, 
it follows from \cite[Theorem 5.9]{vNVW07} that the process is stochastically integrable and the stochastic integral defines a $\mathsf{D}((-A_{p})^\alpha)$-valued random variable. By the embedding from Lemma \ref{l.fdinject}, we are done.
\end{proof}

We may thus define

\begin{defn}
Assume (A), (F) and (G).
A \emph{mild solution} of equation \eqref{eq.scp} is a tupel $((\Omega, \Sigma, \P), \FF, W_H, u)$, where 
$(\Omega, \Sigma, \PP)$ is a probability space endowed with a filtration $\FF$ which satisfies the usual conditions, $W_H$ is 
an $H$-cylindrical Wiener process with respect to $\FF$ and $u = (u(t))_{t \geq}$ is a continuous, $\FF$-progressive, $E$-valued 
process such that for all $t \geq 0$ we have
\begin{equation}\label{eq.mildsol}
u(t) = S(t)u(0) + \int_0^t S(t-s)F(u(s))\, ds + \int_0^t S(t-s)G(u(s))dW_H(s)\, .
\end{equation}
$\P$-almost surely. 

If an initial datum $\xi \in E$ is specified, we say that $((\Omega, \Sigma, \P), \FF, W_H, U)$ is a mild solution 
to the initial value problem corresponding to \eqref{eq.scp}, if it is a mild solution of \eqref{eq.scp} and additionally $\P(U(0) =\xi ) = 1$.
\end{defn}

By the results of \cite[Section 6]{Kmartingale}, the weak and the mild solutions of equation \eqref{eq.scp} coincide. We will thus briefly speak of solutions of equation \eqref{eq.scp}, rather than weak (or mild) solutions.

\subsection{Pathwise uniqueness and strong existence}\label{ss.p}

Typically, when working with stochastically weak solutions, the appropriate uniqueness concept is that of \emph{uniqueness in law}. 
However, we will use the following uniqueness concept:

\begin{defn}
We say that \emph{pathwise uniqueness} holds for \eqref{eq.scp}, if whenever $((\Omega, \Sigma, \P), \FF, W_H, u_j)$ for $j=1,2$
are weak solutions of equation \eqref{eq.scp}, defined on the same probability space and with respect to the same $H$-cylindrical 
Wiener process, satisfying $u_1(0) = u_2(0)$ a.s.\ we have $\P (u_1 = u_2) = 1$.
\end{defn}

For finite dimensional stochastic differential equations, Yamada and Watanabe \cite{yw1} proved that pathwise uniqueness implies uniqueness in law. Moreover, they proved that if solutions exist, 
then they exist strongly, i.e.\ \emph{given} a stochastic basis $(\Omega, \Sigma, \FF, \P)$ on which an $H$-cylindrical 
Wiener process is defined, we can find a continuous, $\FF$-progressive, $E$-valued process $u$ \emph{defined on that stochastic basis}, such that
$((\Omega, \Sigma, \P), \FF, W_H, u)$ is a weak solution. 

These results also generalize to stochastic equations on Banach spaces, see \cite[Section 5]{Kmartingale}.

We will make extensive use of strong existence of solutions in Section \ref{s.existence} to prove existence of solutions to stochastic reaction-diffusion 
systems also for unbounded reaction terms.\smallskip

We should also note that given pathwise uniqueness and existence of solutions for deterministic initial values $\xi$, if follows automatically 
that solutions exist for random initial data $\xi: \Omega \to E$ which are $\cF_0$-measurable. See \cite{Kmartingale} for a proof in 
the infinite dimensional case.

\section{Pathwise uniqueness}\label{s.pathwise}

In this section we prove

\begin{thm}\label{t.uniqueness}
Assume (A), (F) and (G). Then pathwise uniqueness holds for equation \eqref{eq.rds}
\end{thm}

\begin{proof}
We may (and shall) assume that $\rho_{l,m} (t) \geq t$ for $l=1, \ldots, r$ and $m \in \CN$, otherwise replacing $\rho_{l, m}(t)$ with $\rho_{l,m} (t) + t$.

Let us fix $l \in \{1, \cdots, r\}$ and $m \in \CN$.
Similar as in the proof of the classical Yamada-Watanabe theorem \cite{yw1}, we chose a decreasing sequence $a_n \downarrow 0$
such that $a_0 =1$ and 
\[
\int_{a_n}^{a_{n-1}} \frac{1}{\rho_{l,m} (t)} \, dt =n.
\]
This is possible by assumption (G2b). Note that the sequence $(a_n)$, as well as the functions $\psi_n$ and $\varphi_n$ introduced next, depend 
on $l$ and $m$. 
Next we pick functions $\psi_n \in C_c^\infty (\CR)$ such that 
$\mathrm{supp} \, \psi_n \subset (a_n, a_{n-1})$ and 
\[
0 \leq \psi_n (t) \leq \frac{2}{n \rho_{l,m} (t)} \leq \frac{2}{nt}
\quad\mbox{and}\quad \int_{a_n}^{a_{n-1}} \psi_n (t) \, dt = 1.
\]
We define 
\[
\varphi_n (t) := \int_0^{|t|} \int_0^s \psi_n (\tau) \, d\tau\, ds\, .
\]
Then $\varphi_n \in C^\infty(\CR)$ with 
\[
\varphi_n'(t) = \mathrm{sgn}(t) \int_0^{|t|} \psi_n (s)\, ds \quad \mbox{and}\quad 
\varphi_n''(t) = \psi_n (|t|).
\]
We note that
\[
|t| - a_{n-1} = \int_0^{|t|} \one_{(a_{n-1}, \infty)}(s) \, ds \leq \varphi_n (t) \leq |t|,
\]
which implies that $\varphi_n (t) \to |t|$, uniformly on $\CR$. Moreover, $\varphi_n'(t)t = |t| \int_0^{|t|} \psi_n(s)\, ds$
converges to $|t|$ pointwise.\medskip 

After this preparation, we now come to the main part of the proof. Let $u_1:= (u_1^1, \ldots, u_r^1)$ and 
$u_2 = (u_1^2, \ldots, u_r^2)$ be two weak solutions of \eqref{eq.rds}, defined on the same probability space and with respect to 
the same sequence of Brownian motions $(\beta_{l,k})$. Moreover, we assume that $u_1(0) = u_2(0)$ almost surely.

For $m \in \CN$ we define the stopping time $\tau_m$ by 
\[
\tau_m := \inf\big\{ t> 0\, :\, \|u_1(t)\|_E\vee \|u_2(t)\|_E > m \big\},
\]
where we set $\inf\emptyset = \infty$.

For $\lambda >0$, $l \in \{1, \ldots, r\}$ and $x \in \overline{\OO}$ the vector $x^* = (\mu_1, \ldots, \mu_r)$, defined 
by $\mu_l^* := \lambda R(\lambda, A_l)^*\delta_x$ and $\mu_j^* = 0$ for $j \neq l$, is an element of $\mathsf{D}(A^*)$. Thus, 
since $u_1$ and $u_2$ are weak solutions with $u_1(0) = u_2(0)$ almost surely, we have, almost surely,
\[
\begin{aligned}
\big[ \lambda R(\lambda, & A_l)(u^1_l(t\wedge \tau_m) - u_l^2(t\wedge \tau_m)\big](x)\\
& = \int_0^{t\wedge \tau_m}\big[A_l\lambda R(\lambda, A_l)(u_l^1(s) - u_l^2(s))\big](x)\, ds\\
& \quad + \int_0^{t\wedge \tau_m}\big[\lambda R(\lambda, A_l)(F_l(u_1(s)) - F_l(u_2(s)))\big](x)\, ds\\
& \quad + \sum_{k=1}^\infty \int_0^t \one_{[0,\tau_m]}(s)\big[
\lambda R(\lambda, A_l) \big(G_{l,k}(u_l^1(s)) - G_{l,k}(u_l^2(s))\big)\big](x)\, d\beta_{l,k}(s)\, .
\end{aligned}
\]
To simplify notation, we introduce some abbreviations. We write 
\begin{eqnarray*}
\Delta_\lambda u_l(t) &:= & 
\lambda R(\lambda, A_l)\big[ u^1_l(t\wedge \tau_m) - u_l^2(t\wedge \tau_m)\big]\\
\Delta_\lambda F_l(t) & := & \lambda R(\lambda, A_l)\big[(u_1(t\wedge \tau_m)) - F_l(u_2(t\wedge \tau_m))\big]\\
\Delta_{\lambda}G_{l,k}(t) & := & \lambda
R(\lambda, A_l) \big[G_{l,k}(u_l^1(t\wedge \tau_m)) - G_{l,k}(u_l^2(t\wedge\tau_m))\big].
\end{eqnarray*}

From It\^o's formula it follows that, almost surely,
\[
\begin{aligned}
\varphi_n & (\Delta_\lambda u_l(t)(x)) = \int_0^{t\wedge \tau_m} \varphi_n'(\Delta_{\lambda}u_l(s)(x))
\big[ A_l \Delta_\lambda u_l(s)(x) + \Delta_\lambda F_l (s)(x) \big]\, ds\\
& + \half \int_0^{t\wedge \tau_m} \varphi_n''(\Delta_{\lambda}u_l(s)(x)) \sum_{k=1}^\infty \big[ \Delta_\lambda G_{k,l}(s)(x)\big]^2\, ds
+ \,\,\mbox{a martingale}.
\end{aligned}
\]
As this is true for every $x \in \overline{\OO}$, we may integrate over $\OO$ and take expectations. This yields
\begin{equation}\label{eq.ilambdan}
\begin{aligned}
\expect \int_\OO \varphi_n & (\Delta_\lambda u_l (t)(x))\, dx  = 
\expect \int_0^{t\wedge \tau_m}\int_\OO \varphi_n '(\Delta_\lambda u_l(s)(x)) A_l \Delta_\lambda u_l (s)(x)\, dx\, ds\\
& \quad + \expect \int_0^{t\wedge \tau_m}\int_\OO \varphi_n '(\Delta_\lambda u_l(s)(x))\Delta_\lambda F_l (s)(x)\, dx\, ds\\
& \quad + \half \expect \int_0^{t\wedge \tau_m}\int_\OO \varphi_n ''(\Delta_\lambda u_l(s)(x))\sum_{k=1}^\infty \big[\Delta_\lambda G_{l,k} (s)(x)\big]^2\, dx\, ds\\
& =: I_1(n,\lambda) + I_2(n,\lambda) + I_3(n, \lambda).
\end{aligned}
\end{equation}
We proceed in several steps.\medskip 

{\it Step 1:} We estimate $I_1(n, \lambda)$ and let $\lambda \to \infty$.

Noting that $\varphi_n '(\Delta_\lambda u_l(s)) \in H^1(\OO)$ by the chain rule, we find that
\begin{eqnarray*}
I_1(n, \lambda) & = & \expect \int_0^{t\wedge \tau_m} \big(  A_l \Delta_\lambda u_l (s), \varphi_n '(\Delta_\lambda u_l(s)) \big)_{L^2(\OO)}
\, ds\\
& = &  - \expect \int_0^{t\wedge \tau_m} \mathfrak{a}^l \big[ \Delta_\lambda u_l (s) , \varphi_n '(\Delta_\lambda u_l(s))\big]
\, ds \, .
\end{eqnarray*}
Next observe that
\[
\begin{aligned}
&  \quad \mathfrak{a}^l \big[ \Delta_\lambda u_l (s) , \varphi_n '(\Delta_\lambda u_l(s))\big]\\
 & = \int_\OO \sum_{i,j=1}^d a_{ij}^l \cdot (D_i\Delta_\lambda u_l(s))\cdot \varphi_n''(\Delta_\lambda u_l (s)) \cdot (D_j\Delta_\lambda u_l(s))
\, dx\\
& \geq \int_\OO \eta \varphi_n''(\Delta_\lambda u_l (s))\sum_{i=1}^d |D_i\Delta_\lambda u_l(s)|^2 \, dx \geq 0\, ,
\end{aligned}
\]
since $\varphi_n''\geq 0$.
It thus follows that $I_1(n,\lambda) \leq 0$.

We now abbreviate
\[
\Delta u_l (t) := u_l^1(t\wedge \tau_m) - u_l^2(t\wedge \tau_m), \quad
\Delta F_l (t) := F_l (u_1(t\wedge \tau_m)) - F_l(u_2(t\wedge \tau_m)) 
\]
and
\[
\Delta G_{l,k} (t) := G_{l,k}(u_l^1(t\wedge \tau_m)) - G_{l,k}(u_l^2(t\wedge \tau_m)).
\]
For $w \in C(\overline{\OO})$ we have that $\lambda R(\lambda, A_l)w$ converges to $w$ in $C(\overline{\OO})$ as $\lambda \to \infty$. 
Thus, it follows that $\Delta_\lambda u_l(t) \to \Delta u_l(t)$, $\Delta_\lambda F_l (t) \to \Delta F_l (t)$ and  
$\Delta_\lambda G_{l,k}(t) \to \Delta G_{l,k}(t)$ in $C(\overline{\OO})$ as $\lambda \to \infty$ for every $t >0$. 
Inserting the estimate for $I_1(n,\lambda)$ into \eqref{eq.ilambdan} and letting $\lambda \to \infty$ we obtain, using the continuity of  $\varphi_n, \varphi_n'$ and $\varphi_n''$, 
\[
\begin{aligned}
\expect \int_\OO \varphi_n  (\Delta u_l (t)(x))\, dx &
\leq \expect \int_0^{t\wedge \tau_m}\int_\OO \varphi_n '(\Delta u_l(s)(x))\Delta F_l (s)(x)\, dx\, ds\\
& \quad + \half \expect \int_0^{t\wedge \tau_m}\int_\OO \varphi_n ''(\Delta u_l(s)(x))\sum_{k=1}^\infty \big[\Delta G_{l,k} (s)(x)\big]^2\, dx\, ds\\
& =: J_1(n) + J_2(n)
\end{aligned}
\]

{\it Step 2:} We estimate $J_1(n)$ and $J_2(n)$ and let $n \to \infty$.

Since $f_l$ is locally Lipschitz continuous by (F) and since $|\varphi_n'| \leq 1$, it follows that, for a constant $L_m \geq 0$, we have 
\[
J_1(n) \leq L_m \expect \int_0^{t\wedge \tau_m} \int_\OO \sum_{j=1}^r |\Delta u_j (s)(x)|\, dx \, ds \leq L_m \expect \int_0^{t}
 \int_\OO \sum_{j=1}^r |\Delta u_j (s)(x)|\, dx \, ds\, . 
\]
As for $J_2(n)$, using (G2) and the estimate $\varphi_n''(t) = \psi_n (t) \leq \frac{2}{n\rho_{l,m}(t)}$, we see that
\begin{eqnarray*}
 J_2(n)  
& \leq & \half \expect \int_0^{t\wedge\tau_m} \int_\OO \one_{\CR \setminus \{0\}}(\Delta u_l (s)(x)) \frac{2 \sum_{k=1}^\infty \sigma_{l,k, m}(|\Delta u_l(s)(x)|)^2}{n\rho_{l, m} (|\Delta u_l (s)(x)|)} 
\, dx \, ds \\
& =  & \half \expect \int_0^{t\wedge\tau_m} \int_\OO \one_{\CR \setminus \{0\}}(\Delta u_l (s)(x)) \frac{2\rho_{l,m} (\Delta  (s)(x)) dx }{n\rho_{l,m} (|\Delta u_l (s)(x)|)} 
 \, ds \leq \frac{t}{n}
\end{eqnarray*}
Combining these estimates, we find that
\[
\expect \int_\OO \varphi_n  (\Delta u_l (t)(x))\, dx \leq L_m \int_0^{t} \int_\OO \sum_{j=1}^r |\Delta u_j (s)(x)|\, dx \, ds
+ \frac{t}{n} \]
Since $\varphi_n (t) \uparrow |t|$ as $n \to \infty$, it follows upon $n \to \infty$ that 
\begin{equation}\label{eq.estl}
\expect \int_\OO |\Delta u_l (t,x)|\, dx 
\leq  L_m \expect \int_0^{t} \int_\OO \sum_{j=1}^r |\Delta u_j (s)(x)|\, dx \, ds 
\end{equation}

{\it Step 3:} We finish the proof.

As equation \eqref{eq.estl} is true for every $l=1, \ldots, r$, we find, summing up, that 
\[
\expect \int_\OO \sum_{l=1}^r |\Delta u_l (t,x)|\, dx \leq rL_m \int_0^t \expect \int_\OO \sum_{l=1}^r |\Delta u(s)(x)|\, dx\, ds
\]
Thus, by Gronwall's Lemma, 
\[
\expect \int_\OO \sum_{l=1}^r |\Delta u_l (t,x)|\, dx  = \expect \int_\OO \sum_{l=1}^r |u_l^1(t\wedge \tau_m, x) - u_l^2(t\wedge \tau_m, x)|
\, dx \equiv 0 \, .
\]
Since $t \wedge \tau_m \to t$ almost surely, upon $m \to \infty$ it follows that 
\[ \expect \int_\OO \sum_{l=1}^r |u_l^1(t, x) - u_l^2(t, x)|
\, dx = 0
\]
for all $t \geq 0$. As solutions are continuous in $x$, it follows that $u_1(t) = u_2(t)$ almost surely, for every $t \geq 0$. Finally, 
by continuity of the paths, the exceptional set may be chosen independently of $t$, hence $u_1 = u_2$ almost surely, i.e.\ pathwise 
uniqueness.
\end{proof}

\begin{rem}
Note that in the proof of Theorem \ref{t.uniqueness} neither the special structure of the functions $f_l$ in condition (F), nor the linear growth 
assumption on the functions $g_{l,k}$ in condition (G2) was used.  Thus pathwise uniqueness holds already if the functions $f_l$ are locally Lipschitz continuous and the functions $g_{l,k}$ satisfy (G2).
\end{rem}

\section{Positivity}\label{s.positive}

Under additional assumptions on the nonlinearities $f_l$ and $g_{k,l}$, the techniques from Section \ref{s.pathwise} can also be used to prove 
that the solutions of equation \eqref{eq.rds} for positive initial data are almost surely positive. For deterministic reaction-diffusion equations,
the solutions preserve positivity if and only if the reaction term is \emph{quasi positive} \cite{p10}.

\begin{defn}
A map $\Phi = (\Phi_1, \ldots, \Phi_r) : \overline{\OO}\times \CR^r \to \CR^r$ is called \emph{quasi positive} if for every $l=1, \ldots, r$ 
and $s= (s_1, \ldots, s_r)$ with $s_l =0$ and $s_j \geq 0$ for all $j \neq l$ we have $\Phi_l (x, s) \geq 0$ for all $x \in \overline{\OO}$.
\end{defn}

In what follows, we write $s^+ := s \vee 0$ for the positive part of a real number $s$ and $s^-:= (-s)\vee 0$ for the negative part of $s$.

\begin{lem}\label{l.qpos}
Let $\Phi : \overline{\OO}\times \CR^k \to \CR^k$ be quasi positive and locally Lipschitz continuous in the sense that for $n \in \CN$ there 
exists a constant $L_n$ such that
\[
|\Phi_l(x, s_1, \ldots, s_r) - \Phi_l (x, t_1, \ldots, t_n)|\leq L_n \sum_{j=1}^r|s_j-t_j| \quad 
\forall\, x \in \overline{\OO}, s_j,t_j \in [-n,n]\,.
\]
Then for every $l \in \{1, \ldots, r\}$, $x \in \overline{\OO}$ and $s_1, \ldots, s_r \in [-n,n]$ with $s_l \leq 0$ we have
\begin{equation}\label{eq.qpos}
-\Phi_l(x, s_1, \ldots, s_r) \leq L_n \sum_{j=1}^r s_j^-\, .
\end{equation}
\end{lem}

\begin{proof}
Without loss of generality, we assume that $l=1$. We have
\begin{eqnarray*}
-\Phi_1(x, s_1, \ldots, s_r) & = & \Phi_1(x, 0, s_2, \ldots, s_r) - \Phi_1(x, s_1, \ldots, s_n) - \Phi_1(x, 0, s_2, \ldots, s_r)\\
& \leq & L_n |0-s_1| - \Phi_1(x, 0, s_2, \ldots, s_r)\\
& = & L_n s_1^- - \Phi_1(x, 0, s_2, \ldots, s_r)\, .
\end{eqnarray*}
If $\Phi_1(x, 0, s_2, \ldots, s_r) \geq 0$, this already yields \eqref{eq.qpos}. Otherwise, since $\Phi$ is quasi positive, there exists 
an index $2\leq j_0 \leq r$ with $s_{j_0} < 0$. We assume without loss of generality that $j_0 =2$. The same estimate as above shows that
\[-\Phi_1(x, 0, s_2, \ldots, s_r) \leq L_n s_2^- - \Phi_1(x,0,0,s_3, \ldots, s_r) \,.\]
These arguments are now iterated. Doing this, we have proved \eqref{eq.qpos} at some point or, after $r$ iterations, we have proved that 
\[
- \Phi_1(x,s_1, \ldots, s_r) \leq L_n\sum_{j=1}^r s_j^- - \Phi_1(x, 0, \ldots, 0)
\]
which also yields \eqref{eq.qpos}, as $\Phi_1(x, 0, \ldots, 0) \geq 0$ by quasi positivity.
\end{proof}

\begin{thm}\label{t.positive}
Assume (A), (F) and (G). Additionally, assume that $f: \overline{\OO}\times \CR^r \to \CR^r$ is quasi positive and that 
$g_{l,k}(x,0) = 0$ for all $x \in \overline{\OO}$, $l=1, \ldots, r$ and $k \in \CN$. 

If $u= (u_1, \ldots, u_r)$ is a solution of \eqref{eq.rds} with $u_1(0), \ldots, u_r(0) \geq 0$ almost surely, 
then almost surely also $u_l (t,x) \geq 0$ for all $t \geq 0$ and $x \in \overline{\OO}$.
\end{thm}

\begin{proof}
We pick the sequence $a_n$ and the function $\psi_n$ as in the proof of Theorem \ref{t.uniqueness} and define
\[
\varphi_n (t) := \one_{(0,\infty)} \int_0^t\int_0^s \psi_n (\tau)\, d\tau\, ds\, .
\]
Then $\varphi_n \in C^\infty (\CR)$ with $\varphi_n(t) \uparrow t^+$. Moreover, 
\[
\varphi_n'(t) = \one_{(0,\infty)} \int_0^t \psi_n (s)\, ds \quad \mbox{and}\quad \varphi_n''(t) = \one_{(0,\infty)}\psi_n (t)\, .
\]
In particular, $\varphi_n '(t) \uparrow \one_{(0,\infty)}$ for all $t \in \CR$.\medskip 

Given a solution $u$ of \eqref{eq.rds} with $u_1(0), \ldots, u_r(0)\geq 0$ almost surely, we
repeat the computations and estimates in the proof of Theorem \ref{t.uniqueness} with 
$u_1 \equiv 0$ and $u_2=: u$. Doing so, we obtain at the end of Step 1
\[
\begin{aligned}
& \quad \expect  \int_\OO  \varphi_n (- u_l (t\wedge \tau_m,x))\\
 & \leq 
\expect \int_0^{t\wedge \tau_m} \int_\OO \varphi_n'(-u_l(s\wedge \tau_m, x)) \cdot (- f_l(x, u(s\wedge \tau_m, x))\, dx\, ds\\
& \quad + \half \expect \int_0^{t\wedge \tau_m} \int_\OO \varphi_n''(-u_l(s\wedge \tau_m, x))\sum_{k=1}^\infty \big[- g_{l,k}(x, u_l(s \wedge \tau_m, x))
\big]^2 dx\, ds\\
& =: J_1(n) + J_2(n)\, .
\end{aligned}
\]
Note that we do not obtain an extra term for the initial datum, as $\varphi_n (-v_l(x)) = 0$ almost surely. We now estimate the terms
$J_1(n)$ and $J_2(n)$.

Since $\varphi_n'(t) =0$ for $t\leq 0$, the integrand in $J_1(n)$ vanishes, unless $u_l (s \wedge \tau_n, x) < 0$. In that case, since
$f$ is quasipositive and locally Lipschitz continuous, it follows from Lemma \ref{l.qpos} that
\[
J_1(n) \leq \expect \int_0^{t\wedge \tau_m} \int_\OO \varphi_n'(-u_l(s\wedge \tau_m, x)) L_m \sum_{j=1}^r u_j(s\wedge \tau_m, x)^-\, dx\, ds \, .
\]
Concerning $J_2(n)$, since $g_{l,k}(x,0) =0$, we obtain
\begin{eqnarray*}
&&J_2(n)\\
 & = & \half \expect \int_0^{t\wedge \tau_m} \int_\OO 
\varphi_n''(-u_l(s\wedge \tau_m, x))\sum_{k=1}^\infty \big[g_{l,k}(x, 0)- g_{l,k}(x, u_l(s \wedge \tau_m, x))\big]^2 dx\, ds\\
& \leq & \half \expect \int_0^{t\wedge \tau_m} \int_\OO \one_{(0,\infty)}(-u_l (s\wedge \tau_m,x))
\frac{2\rho_l(-u_l (s\wedge \tau_m,x))}{n\rho_l (-u_l (s\wedge \tau_m,x))}\, dx \, ds \leq  \frac{t}{n}
\end{eqnarray*}
With these estimates and $n\to \infty$, it follows that
\[
\expect \int_\OO u_l (t\wedge \tau_m, x)^-\, dx \leq L_m \int_0^t \expect \int_\OO \sum_{j=1}^r u_j(s\wedge \tau_m, x)^-\, dx\, ds 
\]
since $\varphi_n (-u_l(t\wedge \tau_m), x)) \uparrow u_l (t\wedge \tau_m, x)^-$. 

Now the same arguments as in the proof of Theorem \ref{t.uniqueness} show that $\sum_{l=1}^r u_l^- =0$ almost surely, hence, almost surely,
$u_l \geq 0$ for all $l=1, \ldots, r$.
\end{proof}

\section{Existence of solutions}\label{s.existence}

We now prove existence of solutions to equation \eqref{eq.rds}. In the proof, we will use the pathwise uniqueness proved in 
 Section \ref{s.pathwise}. Indeed, by pathwise 
uniqueness, solutions exist strongly, i.e.\ on a given probability space and with respect to a given cylindrical Wiener process. 
This allows us to adopt the strategy  of \cite{KvN12} to our situation and prove existence of solutions via approximation of the coefficients.

We will use the following estimates for deterministic and stochastic convolutions which are a consequence of the factorization technique 
\cite{dpkz}.

\begin{lem}\label{l.convest}
Assume (A), (F) and (G).
Let $(\Omega, \Sigma, \FF, \P)$ be a stochastic basis on which an $H$-cylindrical Wiener process is defined. Moreover, let 
$u = (u_1, \ldots, u_r)$ be a continuous, $\FF$-progressive $E$-valued process. Then for $p>2$
there exists a constant $C(T)$ with $C(T) \to 0$ as $T \to 0$, 
independent of the constants $c_1$ and $c_2$ from (F) and the sequences $\alpha_l$ and $\beta_l$ from (G1)
such that for the stochastic convolution process
\[
 S_l \diamond_l G_l(u_l) := t \mapsto \int_0^t S_l(t-s)G_l (u_l(s)\, dW_{H_l}(s)
\]
we have the estimate 
\[
\expect \| S_l \diamond_l G_l(u_l)\|_{C([0,T]; C(\overline{\OO}))}^p \leq 
C(T)(\|\alpha_l\|_{\ell^2}^p + \|\beta_l\|_{\ell^2}^p\expect \| u_l\|_{C([0,T]; C(\overline{\OO}))}^p) 
\]
and for the deterministic convolution process
\[
S_l\ast K_l (u) := t \mapsto \int_0^t S_l(t-s)K_l (u(s))\, ds
\]
we have the estimate
\[
\expect \| S_l \ast K_l(u_l)\|_{C([0,T]; C(\overline{\OO}))}^p \leq 
C(T)(c_1^p + c_2^p\expect \| u\|_{E}^p)  \, .
\]
\end{lem}

\begin{proof}
We first treat the stochastic convolution. We pick $q$ so large that $\frac{d}{2q} < \half - \frac{1}{p}$ and 
$\alpha \in (\frac{d}{2q}, \half - \frac{1}{p})$. 
Then, for some $\eps > 0$ and $\beta \in (\alpha + \frac{1}{p}, \half)$, we find
\begin{eqnarray*}
&&  \expect \| S_l \diamond_l G_l(u_l)\|_{C([0,T]; C(\overline{\OO}))}^p\\
& \lesssim & \expect \| S_l \diamond_l G_l(u_l)\|_{C([0,T]; \mathsf{D}((-A_{l,q})^\alpha)}^p\\
& \lesssim & T^{\eps p} \int_0^T \expect \| s \mapsto (t-s)^{-\beta} G_l (u_l(s))\|_{\gamma (L^2(0,t; H_l), L^q(\OO))}^p\, ds\\
& \lesssim & T^{\eps p} \int_0^T \expect \Big(\int_0^t(t-s)^{-2\beta} \|G_l (u_l)\|_{\gamma (H_l, L^q(\OO))}^2\, ds\Big)^{\frac{p}{2}}\, dt\\
& \leq & T^{\eps p} \int_0^T \expect \Big(\int_0^t(t-s)^{-2\beta} |\OO|^2(\|\alpha_l\|_{\ell^2} + \|\beta_l\|_{\ell^2} \|u_l\|_{C([0,T; C(\OO))})^2 ds\Big)^{\frac{p}{2}}\, dt\\
& \lesssim & T^{\eps\beta} \int_0^T \expect |\OO|^p (\|\alpha_l\|_{\ell^2} + \|\beta_l\|_{\ell^2} \|u_l\|_{C([0,T; C(\OO))})^p\, dt\\
& \lesssim & T^{\eps\beta} \Big( |\OO|^p (\|\alpha_l\|_{\ell^2}^p + \|\beta_l\|_{\ell^2}^p \expect \|u_l\|_{C([0,T]; C(\OO))}^p)\Big)\, .
\end{eqnarray*}
Here, the first estimate follows from the embedding of Lemma \ref{l.fdinject}, the second is \cite[Proposition 4.2]{vNVW08}, the third uses 
that $L^q(\OO)$ has type 2, whence we have the embedding $L^2(0,t; \gamma (H_l, L^q(\OO)) \inject \gamma (L^2(0,t;H_l), L^q(\OO))$. 
The fourth estimate 
follows from the linear growth of $G_l$, proved in Lemma \ref{l.propg}, the fifth from the fact that $\sup_{t \in [0,T]} \int_0^t(t-s)^{-2\beta}\, ds < \infty$ and the last one from Young's inequality. This proves the first assertion. 

The (easier) proof of the section assertion is similar, using \cite[Proposition 4.2.1]{lunardi} instead of \cite[Proposition 4.2]{vNVW08}.
\end{proof}

We first prove existence of solutions under additional boundedness assumptions on the maps $k_l$ and $g_{l,k}$.

\begin{thm}\label{t.ex1}
Assume (A), (F) and (G). Additionally assume that 
\begin{enumerate}
\item The functions $k_l$ from assumption $F$ are uniformly bounded, i.e.\ there exists a constant $C\geq 0$ such that 
$|k_l(x,s)| \leq C$ for all $(x,s) \in \overline{\OO}\times\CR^k$ and $l=1, \ldots, r$.
\item For $l=1, \ldots, r$, the vector $\beta_l$ from assumption (G1) satisfies $\beta_l \equiv 0$.
\end{enumerate}
Then every $\xi = (\xi_1, \ldots, \xi_r) \in E$, there exists a unique solution 
of equation \eqref{eq.rds} with initial datum $\xi$.
\end{thm}

In the proof of Theorem \ref{t.ex1}, we use the following Lemma which is a reformulation of \cite[Lemma 4.2]{bg99}, see also \cite[Lemma 4.4]{KvN12}.

\begin{lem}\label{l.est1}
Assume that $\mathbf{T} = (T(t))_{t\geq 0}$ is a strongly continuous contraction semigroup on $C(\overline{\OO})$ and that 
$H : C(\overline{\OO}) \to C(\overline{\OO})$ is such that for some constants $a>0$ and $N \in \CN$ we have $\dual{H(u+v)}{\varphi} \leq a(1+\|v\|_\infty)^{2N+1}$ for all $u, v \in C(\overline{\OO})$ and $\varphi \in \partial \|u\|_\infty$.
If for some $\tau>0$, $u_0 \in C(\overline{\OO})$ and two continuous functions $u,v : [0,\tau) \to C(\overline{\OO})$ we have 
\[
u(t) = T(t)u_0  + \int_0^t T(t-s)F(u(s) + v(s))\, ds \quad \forall\, t \in [0,\tau),
\]
Then
\[
\|u(t)\|_\infty \leq \|u_0\|_\infty + \int_0^t a(1+\|v(s)\|_\infty)^{2N+1}\, ds \quad \forall\, t \in [0,\tau)\, .
\]
\end{lem}

\begin{proof}[Proof of Theorem \ref{t.ex1}]
Throughout, we fix a positive time $T>0$. It suffices to prove existence of solutions defined for times $t \in [0,T]$. Indeed, by the results of 
Section \ref{s.pathwise} and the Yamada-Watanabe theorem, any solution exist strongly and any two solutions agree as long as they are both defined.
Thus, solutions defined on bounded time intervals can be glued together to get a solution defined for all times.

We now proceed in two steps.
\medskip 

{\it Step 1} Construction of a maximal solution for equation \eqref{eq.rds}.

To begin with, note that by our additional assumptions (1) and (2), the operators $K_l$ and $G_l$ are uniformly bounded, i.e.\ there exists a 
constant $M \geq 0$ such that 
\[
\| H_l (u_1, \ldots, u_r)\|_{C(\overline{\OO})} \leq M 
\quad\mbox{and}\quad 
\|G_l (u)\|_{\gamma (\ell^2, L^p(\OO))}  \leq M
\]
for all $u_1, \ldots, u_r \in C(\overline{\OO})$ resp.\ $u \in C(\overline{\OO})$. We now also approximate the functions $H_l$ with uniformly 
bounded functions. To that end, we set
\[
h_l^{(n)} (x,s) := \left\{ \begin{array}{ll}
h_l (x,s), & \mbox{if} |s| \leq n \\
h_l(x, n \cdot \sgn s) & \mbox{otherwise}.
\end{array}\right.
\]
for $n \in \CN$ and $l=1, \ldots, r$. We define $H_l^{(n)} : C(\overline{\OO}) \to C(\overline{\OO})$ by $\big[K_l^{(n)}(u)\big](x) := 
h_l^{(n)}(x, u(x))$ for $u \in C(\overline{\OO})$ and set $F_l^{(n)}:= K_l + H_l^{(n)}$. Then, for fixed $n\in \CN$, the map $F^{(n)}:= (F_1^{(n)}, \ldots, F_r^{(n)}):
C(\overline{\OO})^r \to C(\overline{\OO})^r$ is uniformly bounded.

It follows from \cite[Theorem 4.5]{bg99} that there exists a solution $u_n = (u_1^{(n)}, \ldots, u_r^{(n)})$ of the stochastic equation with 
coefficients $A, F^{(n)}$ and $G$ and initial datum $\xi$. 
By pathwise uniqueness these solutions exist strongly. Thus we may assume that the processes $u_n$ 
are defined on a common stochastic basis $(\Omega, \Sigma, \FF, \P)$ and that they are solutions with respect to a common 
$H$-Wiener process $W_H$. Define
\[
\rho_n := \inf \Big\{ t \in [0,T]\,:\, \sum_{l=1}^n \|u_l^{(n)}(t)\|_\infty > n \Big\}
\]
with the convention that $\inf\emptyset = T$.

Since $F^{(n)}(u) = F^{(n+1)}(u)$ for $u \in E$ with $\|u\|_E\leq n$, we can repeat the arguments 
in the proof of Theorem \ref{t.uniqueness} to prove that $u_n(t) = u_{n+1}(t)$ for all $0 \leq t \leq \rho_n \wedge \rho_{n+1}$. This implies in 
particular that $\rho_n \leq \rho_{n+1}$ almost surely. Thus, putting $\rho := \sup_n \rho_n$, we may define a maximal solution of our 
original problem \eqref{eq.rds} by $(u(t))_{t \in [0,\rho)}$, where
\[
u(t) := u_n (t) \quad \mbox{for}\,\, t \in [0, \rho_n]\, .
\]

{\it Step 2} We show that the maximal solution $(u(t))_{t \in [0, \rho)}$ is globally defined.

We prove that for $p>2$ we have 
\begin{equation}\label{eq.lpbound}
 \sup_{n \in \CN}\| u_n\|_{L^p(\Omega; C([0,T]; C(\overline{\OO})^r)} < \infty.
 \end{equation}
By \cite[Corollary 2.6]{KvN12} this will then imply that $\rho \equiv T$ and that $u$ is an element of $L^p(\Omega; C([0,T]; C(\overline{\OO})^r))$. It now easily follows
that $u$ is a solution of our problem \eqref{eq.rds} on the interval $[0,T]$.

To prove \eqref{eq.lpbound}, we define the stochastic processes $v^{(n)}_l$ and $w^{(n)}_l$ for $n \in \CN$ and $l=1, \ldots r$ by
\[
v^{(n)}_l := S_l \ast K_l(u_n) \quad \mbox{and}\quad w^{(n)}_l (t) := S_l\diamond_l G_l(u_l)\, .
\]
Since $H_l$ and $G_l$ are uniformly bounded it follows from Lemma \ref{l.convest} that the processes $v^{(n)}_l$ and $w^{(n)}_l$
are uniformly bounded in $L^p(\Omega; C([0,T]; C(\overline{\OO}))$, say by $C_p$.

Now put $y_l^{(n)} := u_l^{(n)} - v_l^{(n)} - w_l^{(n)}$. Since $u_n$ is a mild solution of the approximate problem, it follows that 
\[
y_l^{(n)}(t) = S_l (t)\xi_l + \int_0^tS_l (t-s) H_l^{(n)}(y_l^{(n)}(s))\, ds + v_l^{(n)}(t) +w_l^{(n)}(t)
\]
almost surely for all $t \in [0,T]$. Note that $H_l^{(n)}$ satisfies assumption (F1) with the same constants as $h_l$. Consequently, 
Lemma \ref{l.fdiss} and Lemma \ref{l.est1} yield for $p>2$ that
\begin{eqnarray*}
&& \expect \|y_l^{(n)}\|_{C([0,T]; C(\overline{\OO}))}^p\\
 & \leq & \expect \sup_{t \in [0,T]} \Big( \|v_l\|_\infty  + \int_0^t a'(1+ \|v_l^{(n)}(s) + w_l^{(n)}(s)\|_\infty)^{2N_l+1}\, ds \Big)^p\\
& \lesssim & \expect  \big(1 + \|v_l\|^p + \|v_l^{(n)}\|_{C([0,T|; C(\overline{\OO}))}^{(2N_l+1)p} +  \|w_l^{(n)}\|_{C([0,T|; C(\overline{\OO}))}^{(2N_{l+1})p}\big)\\
& \leq & 1+ 2C_{(2N_l+1)p} =: M_l
\end{eqnarray*}
It thus follows that 
\begin{eqnarray*}
&& \expect \|u_l^{(n)}\|_{C([0,T]; C(\overline{\OO}))}^p\\
& \lesssim & \expect \|y_l^{(n)}\|_{C([0,T]; C(\overline{\OO}))}^p + \expect \|v_l^{(n)}\|_{C([0,T]; C(\overline{\OO}))}^p + \expect \|w_l^{(n)}\|_{C([0,T]; C(\overline{\OO}))}^p\\
& \leq  & M_l + 2C_p\, .
\end{eqnarray*}
Summing over $l=1, \ldots, r$, \eqref{eq.lpbound} is proved.
\end{proof}

Note that assumption (F2) was not used in the proof of Theorem \ref{t.ex1}. That assumption is used to get rid of the additional boundedness assumption 
on the $G_l$ and $K_l$ in Theorem \ref{t.ex1}. The main tool we use is the following Lemma, which can be found in \cite[Lemma 4.8]{KvN12}.

\begin{lem}\label{l.est2}
Assume that $\mathbf{T} = (T(t))_{t\geq 0}$ is a strongly continuous contraction semigroup on $C(\overline{\OO})$ and that 
$H : C(\overline{\OO}) \to C(\overline{\OO})$ is such that for some constants $a, b >0$ and $N \in \CN$ we have $\dual{H(u+v)- H(v)}{\varphi} \leq 
a(1+\|v\|_\infty )^{2N+1} - b\|u\|_\infty^{2N+1}$ for all $u, v \in C(\overline{\OO})$ and $\varphi \in \partial \|u\|_\infty$.
If $u,v  \in C([0,T];C(\overline{\OO}))$ satisfy
\[
u(t) = \int_0^t T(t-s)H(u(s) + v(s))\, ds \quad \forall\, t \in [0,T],
\]
then
\[
\|u\|_{C([0,T]; C(\overline{\OO})} \leq \big(\frac{4a}{b}\big)^{\frac{1}{2N+1}}(1 + \|v\|_{C([0,T]; C(\overline{\OO}))})\, .
\]
\end{lem}

\begin{thm}\label{t.ex2}
Assume (A), (F) and (G). Let $(\Omega, \Sigma, \FF, \P)$ be a stochastic basis on which an $H$-cylindrical Wiener process $W_H$ with respect 
to $\FF$ is defined.
Then for every initial datum $\xi = (\xi_1, \ldots, \xi_r) \in L^0(\Omega, \cF_0, \P; C(\overline{\OO})^r)$ there exists a unique solution 
$u= (u_1, \ldots, u_r)$ of equation \eqref{eq.rds}.
\end{thm}

\begin{proof}
To proof existence of solutions we approximate the coefficients $G_l$ and $K_l$ with bounded coefficients. To that end, 
we define $k_l^{(n)}$ by setting $k_l^{(n)}(x,s) := k_l(x,s)$ if $\|s\|_1 \leq n$ and $k_l^{(n)}(x,s) := k_l(x, n\|s\|_1^{-1}s)$ otherwise.
The associated composition operator is denoted by $K_l^{(n)}$.
The functions $g_l^{(n)}$ and the associated composition operator are defined analogously. Clearly, $K_l^{(n)}$ and $G_l^{(n)}$ are uniformly 
bounded. Moreover, they are of linear growth with constants independent of $n$.

Setting $F_l^{(n)} := H_l + K_l^{(n)}$ and $F^{(n)}:= (F_1^{(n)}, \ldots, F_r^{(n)})$, it follows from Theorem \ref{t.ex1} that there exists a solution 
$u_n = (u_1^{(n)}, \ldots, u_r^{(n)})$ of equation \eqref{eq.scp} with coefficients $A, F^{(n)}$ and $G^{(n)}:= (G_1^{(n)}, \ldots, 
G_r^{(n)})$. Note that in Theorem \ref{t.ex1}, we have proved existence of solutions only for deterministic initial data. However, as we have 
already remarked in Section \ref{ss.p}, given pathwise uniqueness we obtain existence of solutions for arbitrary $\cF_0$-measurable initial 
data $\xi$.

Proceeding as in the proof of Theorem \ref{t.ex1}, we can glue together the solutions $u_n$ to a maximal solution of our original 
problem and  finish the proof of existence by proving that the approximative solutions $u_n$ for initial data $\xi \in L^p(\Omega, \cF_0, \PP;  E)$ 
are uniformly bounded in $L^p(\Omega; C([0,T]; E))$. This will then prove existence of solutions for $\xi \in L^p(\Omega, \cF_0, \PP;  E)$ . The abstract results of \cite{KvN12} will then yield 
existence of solutions for $\cF_0$-measurable initial 
data $\xi$.\medskip 

To prove such a bound, we introduce the processes
\[
v_l^{(n)} := S_l \ast K_l^{(n)}(u_n) \quad \mbox{and}\quad w_l^{(n)} := S_l \diamond_l G^{(n)}(u_l^{(n)})\, .
\]
Since the maps $K_l^{(n)}$ and $G^{(n)}_l$ are of linear growth with constants independent of $n$, it follows from 
Lemma \ref{l.convest}  that for certain constants $d_1, d_2$ and a constant $C(T)$ which converges to zero as $T \to 0$ we have 
\[
\expect \|w_l^{(n)}\|_{C([0,T]; C(\overline{\OO}))}^p \leq C(T)(d_1 + d_2\expect \|u_l^{(n)}\|_{C([0,T]; C(\overline{\OO})}^p)
\]
and 
\[
\expect \|v_l^{(n)}\|_{C([0,T]; C(\overline{\OO}))}^p \leq C(T)\Big(d_1 + d_2\expect \sum_{l=1}^r\|u_l^{(n)}\|_{C([0,T]; C(\overline{\OO})}^p\Big) \, .
\]
Now we put $y^{(n)}_l := u_l^{(n)} -  v_l^{(n)}  -w_l^{(n)}-S_l(\cdot)\xi_l$. Since $u_n$ is a mild solution of the approximate problem, 
\[
y_l^{(n)}(t) = \int_0^t S_l(t-s)H_l\big[y_l^{(n)}(s) + v_l^{(n)}(t)  + w_l^{(n)}(s) + S_l(s)\xi_l\big]\, ds\, .
\]
Consequently, by Lemma \ref{l.est2},
\begin{eqnarray*}
&& \expect \|y_l^{(n)}\|_{C([0,T]; C(\overline{\OO}))}^p\\
& \lesssim & 1+\expect \|v_l^{(n)}  + w_l^{(n)} + S_l(\cdot)\xi_l\|_{C([0,T]; C(\overline{\OO}))}^p\\
& \lesssim & 1 + \expect \|v_l\|_{C(\overline{\OO})}^p + \expect \|v_l^{(n)}\|_{C([0,T]; C(\overline{\OO}))}^p + \expect \|w_l^{(n)}\|_{C([0,T]; C(\overline{\OO}))}^p\\
& \leq & 1 + \expect \|v_l\|_{C(\overline{\OO})}^p + 2C(T)\Big(d_1 + d_2\expect \sum_{l=1}^r\|u_l^{(n)}\|_{C([0,T]; C(\overline{\OO}))}^p\Big) 
\end{eqnarray*}
Combining with the estimates for $v_l^{(n)}$ and $w_l^{(n)}$ and summing over $l=1, \ldots, r$, we obtain 
\[
\expect \|u_n\|_{C([0,T]; E)}^p
\leq C_0 +C_1(T)\expect \|\xi_l\|_E^p  + C_2(T)\expect \|u_n\|_{C([0,T]; E)}^p
\]
for certain constants $C_0, C_1(T)$ and $C_2(T)$, where $C_1(T), C_2(T)$ converge to 0 as $t\downarrow 0$.
Thus, if we choose $T_0$ so small that $C_2(T_0) < 1$, we obtain
\[
\expect \|u_n\|_{C([0,T_0]; C(\overline{\OO})^r)}^p \leq \frac{1}{1-C_2(T_0)}(C_0 + C_1(T_0)\expect \|\xi_l\|_{E}^p).
\]
Thus we have obtained the uniform $L^p$-bound for small time intervals. Solving with initial datum $\tilde{\xi}:= u(T_0)$ and with cylindrical 
Wiener process $\tilde{W}_H(t) := W_H(t+T_0) -W_H(T_0)$, we obtain solutions $\tilde{u}_n$. By pathwise uniqueness, $\tilde{u}_n(\cdot)
= u_n (\cdot + T_0)$. The same arguments as above yields boundedness of $\tilde{u}_n$ in $L^p(\Omega; C([0,T_0]; E))$, hence 
boundedness of $u_n$ in $L^p(\Omega; C([0,2T_0]; E))$. Inductively, we obtain boundedness in $L^p(\Omega; C([0,T]; E))$ for all $T>0$.
\end{proof}

\section{An example}\label{s.fhn}

In this section, apply our results to specific stochastic reaction diffusion system by presenting a typical example for the reaction term $f$.
More precisely, we consider the situation where $r=2$, rename the variables $u:= u_1$ and $v=v_2$ and consider the 
function $f: \CR^2 \to \CR^2$, defined by
\[
f_1(u,v) := u-u^3 + v \quad\mbox{and}\quad f_2(u,v) := au-bv
\]
where $a$ and $b$ are positive constants. This choice for the reaction term goes back to the fundamental work of Fitzhugh \cite{fitzhugh}
and Nagumo, Arimolo and Yoshizawa \cite{nagumo}. Reaction-diffusion equations involving this reaction term are generic excitable 
systems and appear frequently in applications in chemistry and biology, see e.g.\ \cite{murray}.

This nonlinearity satisfies our assumptions (F). Indeed, setting $h_1(s) := s-s^3$ and $k_1(u,v) := v$, then clearly $k_1$ is Lipschitz continuous 
and of linear growth. It follows from Example \ref{ex.f}
that $f_1$ satisfies (F1) and (F2). Moreover, can put $h_2 \equiv 0$, which trivially satisfies (F1) and (F2), and $h_2(u,v) := u-v$. We should 
also note that the reaction term $f$ is quasi positive, as $f_1(0,v) = v \geq 0$ for $v \geq 0$ and $f_2(u,0) = au \geq 0$ for $u \geq 0$.

Thus, our results immediately yield the following

\begin{thm}\label{t.fhn}
Let $\OO$ be a bounded Lipschitz domain in $\CR^d$, $a_1, a_2 \in L^\infty (\OO; \CR^{d\times d})$ be symmetric and uniformly elliptic.
Moreover, let $R_1, R_2$ be Hilbert-Schmidt operators on $L^2(\OO)$  which are diagonalized by orthonormal bases $(e_k)$ resp.\ $(\tilde{e}_k)$
such that $e_k, \tilde{e}_k \in C(\overline{\OO})$ for all $k$ and $\sum\|R_1e_k\|_\infty^2, \sum\|R_2e_k\|_\infty^2 < \infty$. Finally, let 
$g_1,g_2 : \CR\to \CR$ be locally $\half$-H\"older continuous and of linear growth. Let $(\Omega, \Sigma, \FF, \P)$ be a stochastic basis on which two independent $L^2(\OO)$-cylindrical Wiener processes $W_1$ and $W_2$ are 
defined. 

Then for every $\xi_1, \xi_2 \in L^0(\Omega, \cF_0, \P; C(\overline{\OO}))$, there exist a pathwise unique solution $(u,v)$ of the stochastic 
reaction diffusion system
\[
\left\{
\begin{array}{lll}
d u(t) & = & \big[ \div (a_1\nabla u (t)) + u(t) - u(t)^3 + v(t)\big] dt+ g_1(u(t))R_1dW_1(t)\\

d v(t) & = & \big[ \div (a_2\nabla v (t)) + au(t) - bv(t)\big]dt + g_2(u(t))R_2dW_2(t)\\

\frac{\partial u(t)}{\partial \nu_{a_1}} = \frac{\partial v(t)}{\partial \nu_{a_2}}  & = & 0 \quad \mbox{on } \partial \OO\\
u(0) & = & \xi_1\\
v(0) & = & \xi_2\, .
\end{array}
\right.
\]
If $\xi_1, \xi_2$ are almost surely positive and $g_1(0) = g_2(0) = 0$,
then the solutions $u, v$ are almost surely positive for all times.
\end{thm}

\end{document}